\documentclass[cupthm]{CUP-JNL-JMJ}

\usepackage{latexsym}
\usepackage{graphicx}
\usepackage{multicol,multirow}
\usepackage{amsmath,amssymb,amsfonts}
\usepackage{mathrsfs}
\usepackage{amsthm}
\usepackage{rotating}
\usepackage{appendix}
\usepackage[numbers]{natbib}

\numberwithin{equation}{section}

\theoremstyle{cupplain}
\newtheorem{theorem}{Theorem}[section]
\newtheorem{lemma}[theorem]{Lemma}

\newtheorem{proposition}[theorem]{Proposition}

\theoremstyle{cupdefinition}
\newtheorem{definition}{Definition}[section]

\theoremstyle{cupremark}
\newtheorem{remark}[theorem]{Remark}

\theoremstyle{cupproof}
\newtheorem{proof}{Proof}

\def\A{{\mathbb A}}
\def\F{{\mathbb F}}
\def\N{{\mathbb N}}
\def\Q{{\mathbb Q}}
\def\R{{\mathbb R}}
\def\S{{\mathbb S}}
\def\Z{{\mathbb Z}}

\def\cA{{\mathcal A}}
\def\cC{{\mathcal C}}
\def\cF{{\mathcal F}}
\def\cO{{\mathcal O}}
\def\cR{{\mathcal R}}
\def\cS{{\mathcal S}}
\def\cT{{\mathcal T}}

\def\bfh{{\bf H}}

\def\Spec{{\rm Spec\,}}
\def\sss{{\S}}
\def\spm{{\sss[\pm 1]}}
\def\Se{\frak{ Sets}}
\def\Ses{{\Se_*}}
\def\Hom {{\rm{Hom}}}
\def\id{{\rm id}}
\def\pik{\pi^{\cT}}
\def\spz{{\Spec\Z}}
\def\dop{{\Delta^o}}
\def\crel{{\Se_{2,*}}}
\def\gop{{\Gamma^{o}}}
\def\Ab{\mathfrak{ Ab}}
\def\smod{{\sss-{\rm{Mod}}}}

\def\spzb{{\overline{\Spec\Z}}}
\def\uhom{{\underline \Hom}}

\newcommand{\ie}{{\it i.e.\/}\ }
\newcommand{\eg}{{\it e.g.\/}\ }
\newcommand{\opcit}{{\it op.cit.\/}\ }

\newcommand{\resp}{{\it resp.\/}\ }

\begin{document}

\begin{Frontmatter}

\title[Riemann-Roch for $\spzb$]{RIEMANN-ROCH FOR $\spzb$\thanks{Research supported by The Simons Foudation}}

\author[1,2]{ALAIN CONNES}

\author[3]{CATERINA CONSANI}

\address[1]{\orgname{Coll\`ege de France} 
\orgaddress{\city{Paris}, \state{F-75005 France}}

\email{alain@connes.org}

}

\address[2]{\orgname{IHES}
\orgaddress{\city{Bures-sur-Yvette} \state{91440 France}}

}

\address[3]{\orgdiv{Department of Mathematics} \orgname{The Johns Hopkins University}, 
\orgaddress{\city{Baltimore} \state{21218 USA}}

\email{cconsan1@jhu.edu}

}

\maketitle

\authormark{A. Connes and C. Consani}

\abstract{We prove a Riemann-Roch  theorem of an entirely novel nature for  divisors on  the Arakelov  compactification of the algebraic spectrum of the integers. This result relies on the introduction of  three key concepts: the cohomologies  (attached to a divisor), their \emph{integer} dimension, and Serre duality. These  notions directly extend their classical counterparts for function fields. The Riemann-Roch formula equates the (integer valued) Euler characteristic of a divisor with a slight  modification  of the traditional  expression in terms of the sum of the degree of the divisor  and the  logarithm of 2.  Both the  definitions of the cohomologies and of their dimensions  rely on a universal arithmetic theory over the sphere spectrum that we had previously introduced  using  Segal's Gamma rings. By adopting this new perspective we can parallel Weil's adelic proof of the Riemann-Roch formula for function fields including the use of Pontryagin duality.}


\keywords[\textup{2010} Mathematics subject classification]{14C40; 14G40; 14H05; 11R56; 13F35; 18N60; 19D55}


\end{Frontmatter}

\section{Introduction}
\label{sec1}

The  development of the  theory of curves from  complex Riemann surfaces to algebraic curves   over arbitrary fields, led F. K. Schmidt, in the 1930s, to the first proof of the Riemann-Roch theorem for function fields over  finite fields.  In nowadays terminology, that result involves the integer dimensions of the two cohomologies $H^0(D)$ and $H^1(D)$ of a divisor $D$ on the algebraic curve associated to the function field, as vector spaces over the (finite) field of constants.  The analogy between function fields in one variable over  fields of positive characteristic and number fields suggests, as  pointed out by A. Weil \cite{Weil39},  to  investigate the existence of a Riemann-Roch formula for number fields. Both the notions of a divisor $D$ with its degree, and the finite set underlying $H^0(D)$ are easy to define. These ideas  led S. Lang \cite{Lang} to   an asymptotic  formula  that relates for $\Q$, $\log \# H^0(D)$ and  $\deg D+ \log 2$, when $\deg D\to \infty$.  Another  attempt to  a Riemann-Roch formula for a number field  was promoted two decades ago by G. van der Geer and R. Shoof \cite{vGS} and  more  recently,  in the work of J.B. Bost \cite{Bost}. In this approach one starts from the functional equation in terms of the theta function  and rewrites that equation as a Riemann-Roch formula for the log-theta number, which is thus promoted to the status of a dimension. This view of the log-theta number as  a dimension remains  virtual  for the obvious reason that it outputs real numbers rather than integers. An  interpretation of these numbers as Murray-von Neumann dimensions could greatly improve their status  as dimensions, but  this step has not been achieved so far. \newline
The  most fundamental number field is the field $\Q$ of rational numbers that governs elementary arithmetics through its ring $\Z$ of integers. In this paper we prove a Riemann-Roch theorem of an entirely novel nature for this basic arithmetic structure. This result relies on the introduction of  three key concepts: the two cohomologies $H^\bullet(D)$ (attached to a divisor $D$),  their (integer) dimension, and Serre duality. These three notions directly extend their classical analogues for function fields. 
Our main result is the following theorem.
\begin{theorem} \label{rrspzbintro} 
Let $D$ be an Arakelov divisor on $\spzb$. Then 
\begin{equation}\label{rrforq1intro}
\dim_{\spm}H^0(D)-\dim_{\spm}H^1(D)=\bigg\lceil \deg' D+ \log' 2\bigg\rceil'	-\mathbf 1_L.
\end{equation}
Here,  $\lceil x \rceil'$ denotes the odd function on $\R$ that agrees with the ceiling function on positive reals, and $\mathbf{1}_L$ is the characteristic function of the exceptional set of finite Lebesgue measure which is the union of the intervals $(\log' \frac{3^k}{2},\log' \frac{3^k+1}{2})$:
\end{theorem}
In this formula,  the neperian logarithm  that is traditionally used to define the degree of a divisor $D=\sum_j a_j\{p_j\} + a\{\infty\}$, is replaced by the logarithm in base $3$.\footnote{All throughout the paper, to avoid confusion, we shall nevertheless use the neperian logarithm} This alteration is equivalent to the division by $\log 3$ i.e.  $\deg'(D):=\deg(D)/\log 3$.
Let us now comment on  the key role played by  the number $3$. In the  proof of Theorem~\ref{rrspzbintro}, the ring $\Z$
appears in a new light as a ring of \emph{polynomials} $\sum a_j 3^j$ with coefficients in the absolute base $\spm=\{0,\pm 1\}$. At an elementary level, it is not difficult to check that any integer $n\in \Z$ can be written uniquely as a finite sum of powers of $3$ with coefficients in $\{0,\pm 1\}$. The deep reason behind this  fact is that, for $p=3$ (and only for this rational prime) the Witt vectors with only finitely many non-zero components  
 form a subring  inside  the ring of $p$-adic integers, which turns out to be isomorphic to $\Z$.  In particular, the formula for the addition of integers written as polynomials $\sum a_j 3^j$ coincides with the formula for the addition of Witt vectors over $\F_3$. \newline
Next, we describe  the general  formalism that allows us to give a precise mathematical meaning both to  the base $\spm=\{0,\pm 1\}$ and (to) the two cohomologies $H^\bullet(D)$ with   their integer-valued dimension.  The basic step underlying this formalism is taken by  going beyond the traditional use of abelian groups and rings in algebra using  Segal's $\Gamma$-rings.  In \cite{CCprel, schemeF1} we   developed the fundamentals of a geometric theory where the initial ring $\Z$ of the category of rings is replaced by the sphere spectrum  $\sss$. A similar change of structures is  familiar in homotopy theory,  but in our work we insist on staying at a concrete computational level in the definition of the basic objects, while the link with the $\Gamma$-spaces of homotopy theory \cite{DGM}  enters only when doing homological algebra. The extension of the category of abelian groups is obtained by embedding faithfully and fully this category in the category $\smod$ of pointed covariant functors from finite pointed sets to pointed sets. Given an abelian group $H$ one assigns to a finite pointed set $X$ the pointed set of $H$-valued divisors on $X$. These divisors push forward by summing over the pre-image of a point. In  $\smod$ the traditional notion of ring becomes that of Segal's $\Gamma$-ring, and the  (classical) base $\Z$ is replaced by the $\Gamma$-ring $\sss$ here understood in its most elementary form of identity functor (from finite pointed sets to pointed sets). The base $\sss$ is in fact the most elementary categorical form of  the sphere spectrum in homotopy theory \cite{DGM}.
In analogy with the theory of algebraic curves  over finite fields, we
 have (previously) defined the structure sheaf of the one-point  compactification $\spzb$ of the algebraic spectrum of the integers   by suitably extending, at the archimedean place, the structure sheaf of ${\Spec\Z}$ as a subsheaf of the constant sheaf $\Q$. The  global sections of this sheaf determine the $\sss$-algebra  $\spm$. In  \cite{CCgromov} we have generalized the notion of homology for simplicial complexes, when group coefficients are replaced by $\sss$-modules. In \cite{CCAtiyah} we implemented homological algebra in this context by showing an extension  of the Dold-Kan correspondence which associates a $\Gamma$-space to a short complex of $\sss$-modules.\newline
  Let us now introduce  the cohomologies and their dimensions.
  An Arakelov divisor $D=\sum_j a_j\{p_j\} + a\{\infty\}$ on $\spzb$ determines a compact \emph{subset} $\cO(D)\subset \A_\Q$ of the adeles of $\Q$. This is  the product, indexed by the places of $\Q$, of the abelian group $\Z_p\subset \Q_p$ for each finite prime  $p\notin \{p_j\}_j$, the  abelian groups  $p_j^{-a_j}\Z_{p_j}$, and the interval $[-e^a,e^a]\subset \R$. This last component is implemented by a sub-module of the Eilenberg-MacLane $\spm$-module $H\R$  which functorially encodes the additive structure of the group $\R$. The morphisms of abelian groups used by Weil in the development of the adelic geometry of function fields still retain a meaning in this absolute set-up and they are viewed as  morphisms of $\spm$-modules. We  focus, in particular, on the morphism
 \begin{equation}\label{psiintro}
 \psi: \Q\times \cO(D) \to \A_\Q\quad \psi(q,a) = q+ a\quad \forall q\in \Q,~a\in \cO(D),
 \end{equation}
 where $\Q$ is embedded diagonally  in the adeles. The $\Gamma$-space $\bfh(D)$   associated by the Dold-Kan correspondence to \eqref{psiintro} ($\psi$ is viewed here as a short complex of $\spm$-modules: see appendix~\ref{secB}) provides the absolute incarnation of the Riemann-Roch problem for the divisor $D$. By implementing linear equivalence i.e. the multiplicative action of $\Q^\times$  on $\A_\Q$, one is reduced to consider only the case   $D=a\{\infty\}$, for which  $\ker(\psi)$ is   the $\spm$-module $H^0(D)=\| H\Z\|_{e^a}$ that associates to a finite pointed set $X$ the $\Z$-valued  divisors $\sum_i n_ix_i$, $x_i\in X$,  fulfilling the condition  $\sum_i \vert n_i\vert\leq e^a$ ($\vert\cdot\vert=$  euclidean absolute value).  By definition, $H^0(D)$ is a covariant functor $\Gamma^{o}\longrightarrow \Se_*$ from the small category $\Gamma^{o}$ (a skeleton of the category of finite pointed sets) to pointed sets. It  keeps track of the partially defined addition in $I=\| H\Z\|_{e^a}(1_+)=[-e^a,e^a]\cap \Z$, so that for $n_i\in I$, the sum $\sum n_i$ is meaningful provided  $\sum_i \vert n_i\vert \leq e^a$.
The dimension $\dim_{\spm}H^0(D)$ is  defined to be the smallest number of \emph{linear} generators of $I$. More precisely, a subset $F\subset I$ linearly generates if and only if for every $m\in I$ there exist coefficients $\alpha(f)\in \{-1,0,1\}=\spm(1_+)$ such that $m=\sum \alpha(f)f$, $f\in F$, and $\sum \vert\alpha(f)f\vert\leq e^a$ (see section~\ref{sec3}). The $\spm$-module $H^0(D)$  only depends upon the integer part $n=\lfloor e^a\rfloor$ of $e^a$ and Proposition \ref{dimhzn} determines its dimension to be\footnote{For the  values $n=4,13,40,\ldots, (3^k-1)/2$, every element of $I$ is uniquely written in terms of the generating set $F=\{3^i\mid 0\leq i<k\}$ (Proposition \ref{dimhzn} $(iii)$).}
\begin{equation}\label{dimhznintro}
\dim_{\spm}(\| H\Z\|_n)=\bigg\lceil \frac{\log(2n+1)}{\log 3}\bigg\rceil.
\end{equation}
The cokernel of $\psi$ in \eqref{psiintro}, that is $H^1(D)$, is defined in terms of the $\Gamma$-space $\bfh(D)$. The lack of transitivity of the homotopy relation in non-Kan complexes is dealt with using the classical notion of tolerance relation \cite{P}. Thus $H^1(D)$ is  the couple of the $\spm$-module $H\A_\Q$ and a suitable tolerance relation $\cR$ on it (see  Proposition~\ref{setup1} and  appendix~\ref{secA}). This relation is defined by pairs whose difference belongs to the image of $\psi$. More specifically, $H^1(D)$ defines an object of the category $\Gamma\cT_*$ of tolerance $\sss$-modules (see appendix~\ref{secA}). Its dimension $\dim_{\spm}H^1(D)$ is defined to be the smallest number of linear generators  where a subset $F\subset H\A_\Q(1_+)=\A_\Q$ linearly generates if and only if for every $x\in H\A_\Q(1_+)$ there exists coefficients $\alpha(f)\in \{-1,0,1\}$ such that the pair $(x,\sum \alpha(f)f)$ belongs to the relation $\cR$ (see section~\ref{sec4}).\newline An attentive reader  should readily recognize  that our construction strictly parallels Weil's adelic proof of the Riemann-Roch formula for function fields and, in particular, that our notion of dimension  is a straigthforward generalization of the classical dimension for vector spaces.   For an archimedean divisor  $D=a\{\infty\}$, the dimension of $H^1(D)$ is the same as the dimension of the pair $(H(\R/\Z),\cR)$, where the relation $\cR$ on $H(\R/\Z)(1_+)=\R/\Z$ is given by
 \begin{equation}\label{tolerrelat}
(x,y)\in \cR \iff d(x,y)\leq e^a,
\end{equation}
 for $d$ the translation invariant Riemannian metric of length $1$ on $\R/\Z$.  In Proposition \ref{caseh1} we show that this dimension is the integer 
 \begin{equation}\label{dimh1intro}
\dim_{\spm}(H(\R/\Z),\cR)=\bigg\lceil \frac{-a- \log 2}{\log 3}\bigg\rceil.	
\end{equation}
Theorem~\ref{rrspzbintro}    follows from  \eqref{dimhznintro} when  $\deg D \geq -\log 2$, since in this case $\dim_{\spm}H^1(D)=0$. The case $\deg D<-\log 2$ is proved using \eqref{dimh1intro}.
The symmetry of the term  $\lceil \frac{\deg D+ \log 2}{\log 3}\rceil'$ under the replacement of  $D$ by the divisor $K-D$, with $K:=-2 \{2\}$, is the numerical evidence of a geometric duality holding on the curve $\spzb$ over $\spm$ (see section~\ref{sec5}). This is   precisely stated by the following  (see Theorem~\ref{serredual})
\begin{theorem}\label{serredualintro}
Let $D$ be an Arakelov divisor on $\spzb$ and $K=-2 \{2\}$, with $\deg K=-2 \log 2$. Then there is an isomorphism of $\spm$-modules.
\begin{equation}\label{dualdual}
H^0(K-D)\simeq \uhom_{\Gamma\cT_*}(H^1(D),U(1)_{\frac 14}),
\end{equation}
where $U(1)_{\frac 14}$ coincides with $H^1(K)$,  playing the role of the dualizing module in Pontryagin duality.
\end{theorem}
\noindent The isomorphism $H^0(K-D)\simeq \uhom_{\Gamma\cT_*}(H^1(D),H^1(K))$ is the analogue of Grothendieck's version of Serre's duality \cite{GS} (12.15.1955).
The duality  \eqref{dualdual} derives from   Pontryagin's duality holding for tolerant $\spm$-modules  (Proposition~\ref{pontrj}), once again in analogy with   the corresponding result proven in \cite{Weil39} for curves over  fields of positive characteristic. The divisor $K$ plays the role of the canonical divisor.

The paper is organized as follows, in section \ref{sec2} we adapt to the adelic framework the construction of the $\Gamma$-space $\bfh(D)$   naturally associated to an Arakelov divisor $D$. This gives, in Proposition \ref{setup1}, the cohomologies $H^0(D)$ and $H^1(D)$. We compute the dimension of $H^0(D)$  in section \ref{sec3} (Theorem \ref{mainth}), and of $H^1(D)$ in section \ref{sec4} which concludes with the proof of the main Theorem \ref{rrspzb}. Section \ref{sec5} is devoted to Serre and Pontrjagin dualities. In appendix \ref{secA} we develop the needed generalities on tolerance $\sss$-modules and in appendix \ref{secB} the technical details of the  construction, through the Dold-Kan correspondence, of the  $\Gamma$-space $\bfh(D)$.

\section{The  cohomology $H^\bullet(D)$}\label{sec2}

We recall, from \cite{CCAtiyah}, the construction of the $\Gamma$-space $\bfh(D)$   naturally associated to an Arakelov divisor $D$ on $\spzb$. This space  is the  absolute homological incarnation of the Riemann-Roch problem for the divisor $D$. We formulate this construction  in terms of the adeles of $\Q$.

Let $D=\sum_j a_j\{p_j\} + a\{\infty\}$ be an Arakelov divisor on $\spzb$. This is a formal finite sum with $a_j\in\Z$ and $a\in\R$, where $p_j$ are rational primes in $\Z$ and the symbol $\infty$ stands for the restriction $\vert\cdot\vert: \Q\stackrel{}{\to} \R$ of the  euclidean absolute value $\vert\cdot\vert_\infty$. Let  $\Sigma_\Q$ denote the full set of places $\nu$ of $\Q$.   To  $D$ is naturally associated the following idele $\exp(D)$ of $\Q$ 
\[
\exp(D)_\nu:=\begin{cases}p_j^{-a_j}& \text{if $\nu=\nu_j$}\\
1& \forall \nu\neq \nu_j,\, \forall j\\
e^a& \text{if $\nu=\infty$}.	
\end{cases}
\] 
The equality $\exp(D)(\nu)=\vert \exp(D)_\nu\vert_\nu$ 
 defines a map $\exp(D):\Sigma_\Q\to \R_+^*$ such that $\exp(D)(\nu)\in  {\rm mod}(\Q_\nu)$, $\forall \nu \in \Sigma_\Q$, and $\exp(D)(\nu)=1$,  for   $ \nu\neq \infty,  \nu_j, \forall j$.  To the divisor $D$ corresponds the compact \emph{subset}  of the adeles of $\Q$
\[
\cO(D):=\big\{(a_\nu)\in\A_\Q\mid\vert a_\nu\vert_\nu\le \exp(D)(\nu), \forall \nu\in\Sigma_\Q\big\}\subset \A_\Q.
\]
By definition  $\cO(D)=\cO(D)_f \times \cO(D)_\infty=\prod_\nu \cO(D)_\nu$ with
\[
\cO(D)_\nu=\begin{cases}p_j^{-a_j}\Z_{p_j}& \text{if $\nu=\nu_j$}\\
\Z_\nu & \forall \nu\neq \nu_j,\, \forall j,~ \nu<\infty\\
[-e^a, e^a]& \text{if $\nu=\infty$}.
\end{cases}
\]
In particular, the product $\cO(D)_f=\prod_{\nu\neq \infty} \cO(D)_\nu$ is a compact abelian group. The archimedean component, on the other hand, \ie the real interval $[-e^a, e^a]$, is not, and for this reason we implement the functorial viewpoint to suitably understand $\cO(D)$.  Indeed, to $\cO(D)_\infty$  we associate the  covariant functor 
\[
\| H\R\|_{e^a}: \Gamma^{o}\longrightarrow \Se_*\qquad \| H\R\|_{e^a}(F):=\bigg\{\phi\in H\R(F)\mid \sum_{F\setminus \{*\}} \vert\phi(x)\vert\leq e^a\bigg\}
\]
from the opposite of the Segal category (see \eg \cite{DGM} Chpt. 2 and \cite{CCprel})  to pointed sets, which maps a finite pointed set $F$ to the pointed set of $\R$-valued divisors on $F$ vanishing on the base point and whose total mass $\sum_{F\setminus \{*\}} \vert\phi(x)\vert$ is bounded by $e^a$. Covariant functors $\Gamma^{o}\longrightarrow \Se_*$ and their natural transformations determine the category $\Gamma\Ses$ of $\Gamma$-sets (aka $\sss$-modules). In particular, the Eilenberg-MacLane functor $H$  encodes an abelian group $A$ as the covariant functor $HA:\Gamma^{o}\longrightarrow \Se_*$  that associates to  $F$ the pointed set of $A$-valued divisors on $F$ vanishing on the base point. Functoriality holds by taking  sums  on the inverse images of a point. The functor $HA$ encodes the addition on $A=HA(1_+)$. In general, let $\sigma\in \Hom_\gop(k_+,1_+)$ with $\sigma(\ell)=1$ $\forall \ell\neq *$ and
\begin{equation}\label{partmaps}
\delta(j,k)\in \Hom_\gop(k_+,1_+), \quad \delta(j,k)(\ell):=\begin{cases} 1 &\text{if $\ell=j$}\\ * &\text{if $\ell\neq j$}.
\end{cases}
\end{equation} 
Given an $\sss$-module $\cF$ and  elements $x,x_j\in \cF(1_+)$, $j=1,\ldots, k$, one writes 
\begin{equation}\label{defnsum}
x=\sum_j x_j\iff \exists z	\in \cF(k_+)~\text{s.t.}~  \cF(\sigma)(z)=x,~ \cF(\delta(j,k))(z)=x_j,\ \forall j.
\end{equation}
  For $A=\R$ the key point is that this general construction respects the bound on the total mass while encoding the addition. One easily sees that $\| H\R\|_{e^a}$ is an $\spm$-module (an ``($\spm\wedge -)$-algebra'' as in  \cite{DGM} 2.1.5), where  $\spm:  \Gamma^{o}\longrightarrow \Se_*$, $\spm(F):=\{-1,0,1\}\wedge F$ is the spherical monoid algebra of the multiplicative monoid $\{\pm 1\}$. In particular, the inclusion functor $\| H\R\|_{e^a}\longrightarrow H\R$ determines $\| H\R\|_{e^a}$ as  a sub $\spm$-module of $H\R$, where $\sss: \Gamma^{o}\longrightarrow \Se_*$ is the identity functor. 

Given two $\sss$-modules $\cF_j$, $j=1,2$, one defines their  product as the functor
\[
\cF_1 \times \cF_2:\Gamma^{o}\longrightarrow \Se_*, \qquad (\cF_1 \times \cF_2)(F)=\cF_1(F)\times \cF_2(F)
\]
 with the base point of $\cF_1(F)\times \cF_2(F)$ taken to be $(*,*)$. For abelian groups $A,B$ one has a canonical isomorphism $H(A\times B)\simeq HA\times HB$. In particular, 
 the morphism of  addition in adeles $\alpha:\Q\times \A_\Q \to \A_\Q$, $\alpha(q, a) = q+ a$     (using the diagonal embedding of $\Q$ in $\A_\Q$) determines a morphism of $\spm$-modules $H\alpha: H\Q\times H\A_\Q \longrightarrow H\A_\Q$. Next proposition shows how this functorial construction is well combined with the adelic formalism.
 
\begin{proposition} \label{setup}
\begin{enumerate}
\item[(i)] The $\spm$-module 
\[
H\cO(D):=H(\cO(D)_f)\times \| H\R\|_{e^a}
\]
determines, canonically, a sub $\spm$-module $\iota:H\cO(D)\longrightarrow H\A_\Q$.
\item[(ii)] The restriction  of $H\alpha$ to $H\Q\times H\cO(D)$ defines a morphism of $\spm$-modules \begin{equation}\label{psi}
 \psi: H\Q\times H\cO(D) \longrightarrow H\A_\Q.
 \end{equation} 	
\end{enumerate}
\end{proposition}
\proof $(i)$~With $\A_\Q=\A_f\times \R$, let $\iota_f:\cO(D)_f\to \A_f$ be the inclusion of abelian groups.  Then  $\iota:H\cO(D)\to H\A_\Q$ is the product of $H\iota_f$ and the inclusion $\| H\R\|_{e^a}\to H\R$.\newline
 $(ii)$~By composition one obtains $\psi= H\alpha \circ (H\id_\Q \times \iota)$. \endproof 
 
 The morphism $\psi$ in \eqref{psi} is obtained by restricting  $H\alpha$  to the (adelic) divisor $D$, where $\alpha$ is a morphism of abelian groups.   The theory developed in \cite{CCAtiyah}  associates to   $\psi$ a $\Gamma$-space $\bfh(D)$, by implementing the Dold-Kan correspondence in the special case of a short complex of abelian groups and its restriction to a sub-$\sss$-module. The $\Gamma$-space $\bfh(D)$  supplies the absolute cohomology of the divisor $D$. We recall this construction in appendix \ref{secB}. Moreover, we review how the generalized homotopy $\pik_\ast(\bfh(D))$, where $\cT$ is the category of tolerance relations defined in \ref{gamma-t},  is meaningful in this context supplying, with $\pik_1(\bfh(D))$   and $\pik_0(\bfh(D))$,  the (tolerant) $\sss$-modules descriptions  of \resp  the  kernel and the cokernel  of $\psi$. In other words, one obtains the following result

\begin{proposition} \label{setup1}
\begin{enumerate}
\item[(i)] The kernel  of $\psi$ is the  $\spm$-module:
\begin{equation}\label{kerpsi}
H^0(D) =(H\id_\Q \times \iota)^{-1}( H\ker( \alpha))\simeq \iota^{-1} H\Q =\Vert H^0(\spz ,\cO(D)_f)\Vert_{e^a}. 
\end{equation}
\item[(ii)] The cokernel of $\psi$ is the  $\spm$-module $H^1(D) = (H\A_\Q, \cR)$  endowed with the relations
\begin{equation}\label{cokerpsi}
 \cR_k:=\{ (x,y)\in H\A_\Q(k_+)\times  H\A_\Q(k_+)\mid x-y\in {\rm Image}\, \psi(k_+)\}.
\end{equation}
\item[(iii)] Both $H^0(D)$ and $H^1(D)$ only depend, up to isomorphism, on the linear equivalence class of the divisor $D$.
\end{enumerate}
\end{proposition}
\proof We refer to appendix \ref{sectproofprop}.
\endproof  

\section{The dimension of $H^0(D)$}\label{sec3}
  
In this section we consider, for each integer $n>0$, the $\spm$-module $\Vert H\Z\Vert_n$ 
\[
\| H\Z\|_{n}: \Gamma^{o}\longrightarrow \Se_*\qquad \| H\Z\|_{n}(F):=\{\phi\in H\Z(F)\mid \sum_{F\setminus \{*\}} \vert\phi(x)\vert\leq n\}
\]
and evaluate its dimension.  It follows from the $\spm$-module structure, that \eqref{defnsum} holds for a sum of the form 
\[
\sum_F \alpha_j j\in [-n,n]\cap \Z=\Vert H\Z\Vert_n(1_+) \quad \forall \alpha_j\in \{-1,0,1\}~~\text{s.t.}~ \sum_F \vert \alpha_j j\vert \leq n.
\]
By applying Definition \ref{generator}, 	a subset $F\subset [-n,n]\cap \Z=\Vert H\Z\Vert_n(1_+)$ generates\footnote{For simplicity we use from now on the term ``generates" for ``linearly generates"} $\Vert H\Z\Vert_n(1_+)$ if and only if for every element  $x\in \Vert H\Z\Vert_n(1_+)$ there exists coefficients $\alpha_j\in \{-1,0,1\}$, $j\in F$, such that $x=\sum_F \alpha_j j$ and $\sum_F \vert \alpha_j j\vert \leq n$. The {\em dimension} $\dim_{\spm}\Vert H\Z\Vert_n$ is the minimal cardinality of a generating set.
	
The explicit computation of  $\dim_{\spm}\Vert H\Z\Vert_n$ follows from the next two statements.
	
\begin{lemma}\label{bijectivecase} For $k\in\N$,  let $n=(3^k-1)/2$. Consider the subset $F:=\{3^i,0\leq i\leq k-1\}\subset [-n,n]$. Then $F$ is a generating set  and the map 
\begin{equation}\label{maptheta}
	\theta:\{-1,0,1\}^{\times k}\to [-n,n]\cap \Z, \qquad \theta((\alpha_i)):=\sum_{i=0}^{k-1} \alpha_i 3^i
	\end{equation}
is bijective. Moreover $\sum_{i=0}^{k-1} \vert \alpha_i 3^i\vert \leq n$,	 for all $\alpha_i\in \{-1,0,1\}^{\times k}$.
	\end{lemma}
\proof First of all, note that for all $\alpha\in \{-1,0,1\}^{\times k}$ one has
\[
\sum_{i=0}^{k-1} \vert \alpha_i 3^i\vert \leq \sum_{i=0}^{k-1} 3^i= n.
\]
This shows that for any $\alpha \in \{-1,0,1\}^{\times k}$ one has $z=(\alpha_i 3^i)\in \Vert H\Z\Vert_n(k_+)$, and, as in \eqref{defnsum}, $\Vert H\Z\Vert_n(\sigma)(z)=\theta((\alpha_i))$.
 We prove the bijectivity of $\theta$ by induction on $k$. It is clear for $k=1$. For $k=2$, $n=4$ and one writes the set $[-4,4]\cap \Z$ as the disjoint union of three subsets as follows
\begin{align*}
[-4,4]\cap \Z&=\{-4,-3,-2\}\cup \{-1,0,1\}\cup \{2,3,4\}=\\
&=(-3+\{-1,0,1\})\cup \{-1,0,1\}\cup (3+\{-1,0,1\}).
\end{align*}
This shows  that  $\theta$ is bijective. Next, by assuming to have shown the bijectivity up to $k-1$, we deduce it for $k$. Indeed, with $n=(3^k-1)/2$ let $m=(3^{k-1}-1)/2$. The induction hypothesis ensures that every element of the set $[-m,m]\cap \Z$ can be uniquely written as a sum $\sum_{i=0}^{k-2} \alpha_i 3^i$. We  divide the set $[-n,n]\cap \Z$ into three disjoint parts of the form 
\[
J_{-1}:=[-n,-n+2m]\cap \Z, \quad J_0:=[-m,m]\cap \Z, \quad J_{1}:=[n-2m,n]\cap \Z.
\]
One has $n-2m=m+1$ and thus the union of these three disjoint sets is $[-n,n]\cap \Z$. Moreover one has $J_{1}=(n-m)+J_0$ with $n-m=3^{k-1}$, and similarly $J_{-1}=J_0-3^{k-1}$. Thus, given $x\in [-n,n]\cap \Z$ one determines uniquely the coefficients $\alpha_i\in \{-1,0,1\}$ such that $\theta((\alpha_i))=x$. Indeed,  one uses the above partition in three intervals to determine the coefficient $\alpha_{k-1}$ and then one applies the induction hypothesis to determine the  others. \endproof 
\begin{remark}\label{f3} \begin{enumerate}
\item The conceptual explanation of Lemma \ref{bijectivecase}  derives from the following peculiar property of the Teichm\"uller lift $\tau:\F_3\to \Z_3$. One has $\tau(\F_3)=\{-1,0,1\}\subset \Z\subset \Z_3$ and $\tau$ extends to a canonical map $\tilde \tau :W_k(\F_3)\to \Z$ from Witt vectors to $\Z$  
\[
\xi=(\xi_j)\in W_k(\F_3)\mapsto \tilde \tau(\xi):= \sum_0^{k-1} \tau(\xi_j) 3^j\in \Z.
\]
This map coincides with $\theta$ in \eqref{maptheta} and  one can compute the partially defined sums \eqref{defnsum} in $\Vert H\Z\Vert_n$ using the addition in $W_k(\F_3)$. The prime $p=3$ is the only prime such that the subring $\Z\subset \Z_p=W(\F_p)$ coincides with  the ring of Witt vectors with finitely many non-zero components. When $p>3$ one has $\tau(\F_p)\not\subset \Z $, while for $p=2$ the integer $-1\in \Z\subset \Z_2$ is  the Witt vector  whose components are all equal to $1$. In general the Witt vectors with finitely many non-zero components do not even form a subgroup of the additive group of Witt vectors.
\item The Teichm\"uller lift $\tau$, as a morphism of multiplicative pointed monoids, induces a morphism $\spm\to H\Z$ of $\sss$-algebras.
\item The ordering of the natural numbers encoded by  $\theta$ coincides with the  lexicographic ordering of the coefficients  $(\alpha_i)\in\{-1,0,1\}^{\times k}$ \[ \theta((\alpha_i))< \theta((\beta_i))\iff \exists j ~\text{s.t.}~\alpha_j<\beta_j~~\text{and}~~ \alpha_\ell=\beta_\ell \quad \forall \ell>j.
\]
\end{enumerate}
\end{remark}

We recall that the ceiling function $\lceil x\rceil$ associates to a positive real number $x$ the smallest integer $n\geq x$.

\begin{proposition}\label{dimhzn}
\begin{enumerate}
\item[(i)] For any integer $n \geq 0$ one has 
\begin{equation}\label{dimhzn1}
\dim_{\spm}\Vert H\Z\Vert_n=\bigg\lceil \frac{\log(2n+1)}{\log 3}\bigg\rceil.
\end{equation}
\item[(ii)] For $n\notin \{2,5\}$, there exists a generating subset $F\subset \{1,\ldots,n\}$ of cardinality $\#F=\lceil \frac{\log(2n+1)}{\log 3}\rceil$, with $\sum_{j\in F} j=n$, such that the following map surjects onto $[-n,n]\cap \Z$
\[
	\theta:\{-1,0,1\}^{\times(\#F)}\to [-n,n]\cap \Z, \qquad \theta((\alpha_i)):=\sum_{j\in F} \alpha_j j.
\]
		\end{enumerate}
\end{proposition}
\proof $(i)$~Set $\kappa(n):=\dim_{\spm}\Vert H\Z\Vert_n$. Since the cardinality of $[-n,n]\cap \Z$ is $2n+1$ and the cardinality of $\{-1,0,1\}^{\times\#F}$ is $3^{\#F}$, it follows that if $F$ is a generating set one must have $3^{\#F}\geq 2n+1$. Thus one has $\#F\geq \frac{\log(2n+1)}{\log 3} $, and since $\#F$ is an integer one gets $\#F\geq \lceil \frac{\log(2n+1)}{\log 3}\rceil $. This shows that $\kappa(n)\geq \lceil \frac{\log(2n+1)}{\log 3}\rceil $. It remains to prove that for each $n\ge 0$ one can find a generating set $F$ with $\# F=\lceil \frac{\log(2n+1)}{\log 3}\rceil $. One uses Lemma \ref{bijectivecase}, and distinguishes several cases starting with the easiest one. Set $F(k):=\{3^i,0\leq i\leq k-1\}$ and let $E\subset \N$ be the subset
\begin{equation}\label{exceptionale}
	E=\bigg\{3^\ell+\frac{1}{2} \left(3^m-1\right)\mid m\in \N, \ell<m \bigg\}.
\end{equation}
We list the first elements of $E$ as follows
\[
2,5,7,14,16,22,41,43,49,67,122,124,130,148,202,365,367,373,391,445,607\ldots
\]
 First, assume  that $n\notin E$. Then let $m$ be the largest integer such that $3^m\leq 2n+1$. If  $3^m= 2n+1$, then  Lemma \ref{bijectivecase} provides one with 
a generating set of cardinality $m=\frac{\log(2n+1)}{\log 3}$. Assume now that $3^m< 2n+1$. Then,  the map $\theta$ for $F(m)$  covers the set $[-q,q]\cap \Z$, where $q=\frac{1}{2} \left(3^m-1\right)$. Thus by adjoining to $F(m)$  the element $n-q$ one  covers $[-n,n]\cap \Z$. More precisely,  let 
$F=F(m)\cup \{n-\frac{1}{2} \left(3^m-1\right)\}$, then we show that $F$ fulfills the conditions in $(ii)$. One has $\# F = \# F(m)+1$ since the additional element $n-\frac{1}{2} \left(3^m-1\right)$ does not belong to $F(m)$ as, by hypothesis, $n\notin E$. Since $\# F(m)=m$, $\#F=m+1=\lceil \frac{\log(2n+1)}{\log 3}\rceil$. The sum of elements in $F(m)$ is $\frac{1}{2} \left(3^m-1\right)$, so the sum of elements in $F$ is equal to $n$. Next, we show that  $\theta$ is surjective. Its image contains the three sets $\theta(F(m))$, $\theta(F(m))+n-\frac{1}{2} \left(3^m-1\right)$, and $\theta(F(m))-(n-\frac{1}{2} \left(3^m-1\right))$.  Using Lemma \ref{bijectivecase} one obtains, with $q=\frac{1}{2} \left(3^m-1\right)$
\begin{align*}
&\theta F(m)=[-q,q]\cap \Z, \quad \theta F(m)+n-q=[n-2q,n]\cap \Z,\\ &\theta F(m)-(n-q)=[-n,-n+2q]\cap \Z.
\end{align*}
One has $3(2q+1)=3^{m+1}>2n+1$. This inequality  prevents the three subsets from being disjoint, thus the upper limit $q$ of $\theta F(m)$ is greater or equal to the lower limit $n-2q$ of the interval $\theta F(m)+n-q$. Thus $\theta$ is surjective. \newline
The next case is for $n\in E$ of the form $n=1+\frac{1}{2} \left(3^m-1\right)$, for some $m$ ($\ell=0$ in \eqref{exceptionale}). We assume that $m>2$ (this excludes the cases $n=2$ for $m=1$, and $n=5$ for $m=2$). Let 
\[
F:=\{3^k\mid 0\leq k\leq m-2\}\cup \{2,3^{m-1}-1\}.
\]
(This choice avoids the repetition of the element $1=n-q$ while keeping the sum of elements of $F$ equal to $n$).
We show that $F$ fulfills the conditions in $(ii)$. We have $\# F=m-1+2=m+1$ since $3^{m-1}-1 \notin \{1,2\}$. Moreover $2n+1=3^m+2$,  thus $\#F=m+1=\lceil \frac{\log(2n+1)}{\log 3}\rceil$. The sum of terms in $F$ is equal to one plus the sum of the powers of $3$ up to $3^{m-1}$ and thus is equal to  $n$. Let us now show that the map $\theta$ is surjective. By construction $\theta(F)$ is the union of  nine intervals obtained, with $q'=\frac{1}{2} (3^{m-1}-1)$, by adding to the set $[-q',q']\cap \Z$ the terms in $\{2,3^{m-1}-1\}$ with coefficients in $\{-1,0,1\}$. With the term $3^{m-1}-1$ one covers the set $[-q+1,q-1]\cap \Z$, where $q=\frac{1}{2} \left(3^{m}-1\right)$ as the union of the three sets
\[
[-q',q']\cap \Z, \ \ ([-q',q']+3^{m-1}-1)\cap \Z, \ \ ([-q',q']-(3^{m-1}-1))\cap \Z.
\] 
One has $q'+3^{m-1}-1=n-2$, and thus it  remains to cover the two elements $n-1,n$ and their opposite. This is done using the set
$([-q',q']+3^{m-1}-1 +2)\cap \Z$ whose upper limit is $n$, and appealing to the fact that since $q'>1$ this interval contains $n-1$.\newline
Finally, we consider the case  $n\in E$  of the form $n=3^\ell+\frac{1}{2} \left(3^m-1\right)$, with $m\in \N, ~0< \ell<m $. Here, the additional element $n-\frac{1}{2} \left(3^m-1\right)$ does  belong to $F(m)$ so the proof of the first case considered above does not apply since $3^\ell$ appears twice in $F(m)\cup \{n-\frac{1}{2} \left(3^m-1\right)\}$. We replace this double occurrence of $3^\ell$ by the two distinct elements $3^\ell\pm 1$. This gives
\[
F=(F(m)\setminus \{3^\ell\})\cup \{3^\ell-1,3^\ell+1\}.
\]
By construction  $\# F(m) +1=m+1$. Furthermore one has:  
$n=3^\ell+\frac{1}{2} \left(3^m-1\right)$, $2n+1=3^m+2\times  3^\ell$ and since $0<\ell<m$ one gets $\#F=m+1=\lceil \frac{\log(2n+1)}{\log 3}\rceil$. The sum of elements of $F$ is the same as the sum of elements of $F(m)$ plus $3^\ell$ which is equal to $n$. It remains to show that, for such $F$,  $\theta$  is surjective. 
We first deal with the set $G(\ell):=F(\ell)\cup \{3^\ell-1,3^\ell+1\}$ and show that the map $\theta_{G(\ell)}$, for this set, surjects onto the interval $[-t(\ell),t(\ell)]\cap \Z$, where $t(\ell)=\sum_{G(\ell)}j=3^\ell+\frac{1}{2}(3^{\ell+1}-1)$. Lemma \ref{bijectivecase} shows that the image of $\theta_{F(\ell)}$ is the set $J\cap \Z$, with $J=[-\frac{1}{2}(3^\ell-1),\frac{1}{2} (3^\ell-1)]$. To show the surjectivity of $\theta_{G(\ell)}$ it is enough to show that the following intervals cover $[0,t(\ell)]$:
\[
J, \quad J+3^\ell-1, \quad J+ 3^\ell+1, \quad J+2\times 3^\ell.
\]
The upper limit of $J+2\times 3^\ell$ is $\frac{1}{2} (3^\ell-1)+2\times 3^\ell=t(\ell)$, its lower limit is equal to the upper limit of $J+ 3^\ell+1$. The lower limit of $J+3^\ell-1$ is $3^\ell-1-\frac{1}{2}(3^\ell-1)= \frac{1}{2}(3^\ell-1)$ which is the upper limit of $J$. Moreover, the length of $J$ is $3^\ell-1\geq 2$ so that two translates of $J$ of the form $J+a$, $J+a +2$ necessarily overlap. This shows that the map $\theta_{G(\ell)}$  surjects onto the interval $[-t(\ell),t(\ell)]\cap \Z$. Thus  $\theta_{G(\ell)}$ is the symmetric enlargement of the interval $\theta(F(\ell+1))$ obtained by adding $3^\ell$ to the upper limit. Then, as in the proof of Lemma \ref{bijectivecase}, one obtains by induction that when one adjoins the higher powers $3^{\ell'}$, where $\ell<\ell'<m$, one achieves the enlargement of the interval $\theta(F(m))$ obtained by adding $3^{\ell}$ to the upper limit. This shows that the map $\theta$ for $F$ is surjective. In fact the above arguments prove $(ii)$ and $(i)$, except that one has to take care of the special cases $n=2,5$. For $n=2$,  $\{1,2\}$ is a generating set but the sum of its terms is $>2$. For $n=5$,   $\{1,2,3\}$ is a generating set but the sum of its terms is $>5$.  \endproof

We are now ready to determine the dimension of the $\spm$-module $H^0(D)$, for any Arakelov divisor $D$. It is a non-decreasing function of $\deg D$ but the formula for this dimension drops by $1$ on an exceptional set $L$.   It is the open set defined as the union of open intervals as follows 
\begin{equation}\label{excL}
	L+ \log 2=\bigcup_{k\in \N} (k\log 3,k\log 3+\epsilon_k), \ \ \epsilon_k:=\log(1+3^{-k}).
\end{equation}
It has finite Lebesgue measure
\[
 \vert L\vert=\sum_{k\geq 0}  \epsilon_k=\log \left(-1;\frac{1}{3}\right)_{\infty }= 1.14099\ldots
\]
where $\left(-1;\frac{1}{3}\right)_{\infty }=\prod_{n=0}^\infty (1+3^{-n})$ is the $-1$ Pochhammer symbol at $\frac 13$.

\begin{theorem} \label{mainth}
Let $D$ be an Arakelov divisor  on $\spzb$. If $\deg D\geq -\log 2$, then
\begin{equation}\label{rrforq}
\dim_{\spm}H^0(D)=\bigg\lceil \frac{\deg D+ \log 2}{\log 3}\bigg\rceil	-{\mathbf 1}_L
\end{equation}
where ${\mathbf 1}_L$ is the characteristic function of the open set $L$. 
\end{theorem}
\proof Let $a=\deg D$ and  $n=\lfloor\exp(\deg(D)\rfloor$, then  $H^0(D)=\Vert H\Z\Vert_n$.  By Proposition \ref{dimhzn} one gets 
\[
\dim_{\spm}H^0(D)=\bigg\lceil \frac{\log(2n+1)}{\log 3}\bigg\rceil, \qquad n:=\lfloor\exp(a)\rfloor.
\]
Let us compare this formula with \eqref{rrforq}. Assume first that $a\in L$, thus there exists $k\in\N$ such that $a+\log 2\in (k\log 3,k\log 3+\epsilon_k)$. Then, $2 \exp a\in (3^k, 3^k+1) $ and with $n=\lfloor \exp a\rfloor$,  one has $n=\frac{1}{2}(3^k-1)$. Thus  $2n+1=3^k$ and 
 $\lceil \frac{\log(2n+1)}{\log 3}\rceil=k$. Moreover,  $\frac{a+\log 2}{\log 3}\in (k,k+\frac{\epsilon_k}{\log 3})$, where $0<\frac{\epsilon_k}{\log 3}<1$, hence one has $\lceil \frac{a+ \log 2}{\log 3}\rceil=k+1$, so that \eqref{rrforq} holds since ${\mathbf 1}_L(a)=1$. Now, assume that $a\notin L$,  equivalently that $\frac{a+\log 2}{\log 3}\in [k+\frac{\epsilon_k}{\log 3},k+1]$ for some integer $k\geq 0$ (since $a\geq- \log 2$ by hypothesis). Then  $\lceil \frac{a+ \log 2}{\log 3}\rceil=k+1$. One has 
 \[
 2 \exp a\in [\exp(k\log 3)(1+3^{-k}),\exp((k+1)\log 3)]=[3^k+1,3^{k+1}].
 \]
 In this case $n=\lfloor\exp a\rfloor\in[\frac{1}{2} (3^k+1),\frac{1}{2}(3^{k+1}-1)]$ and thus $2n+1\in [3^k+2,3^{k+1}]$.
 This implies that $\frac{\log(2n+1)}{\log 3}\in (k,k+1]$ and hence that  $\lceil \frac{\log(2n+1)}{\log 3}\rceil=k+1$. In this case \eqref{rrforq} holds since ${\mathbf 1}_L(a)=0$. \endproof 

\section{The dimension of $H^1(D)$}\label{sec4}
 
 Let $(A,d)$ be an abelian group endowed with a translation invariant metric $d$. For $\lambda\in\R_{>0}$, we shall refer to $(A,d)_\lambda$ as the associated object  of the category $\Gamma\cT_*$ as in Proposition \ref{defnA}. The $\sss$-module structure determines the addition in $A=(A,d)_\lambda(1_+)$, as well as the action of $\spm$ on it, where for $x\in A=(A,d)_\lambda(1_+)$, $-x$ is the additive inverse of $x$. The metric $d$ determines the tolerance relation $\cR$ on $A=(A,d)_\lambda(1_+)$.
 Next result computes the dimension (Definition \ref{generator}) of the tolerant  $\spm$-module $(\R/\Z,d)_\lambda$.

 \begin{proposition}\label{caseh1}
 Let $\lambda\in\R_{>0}$, and  $U(1)_\lambda=(U(1),d)_\lambda$ where $U(1)$ is the abelian group $\R/\Z$ endowed with the canonical metric $d$ of length $1$.	Then 
 
 \begin{equation}\label{dimh1}
\dim_{\spm}U(1)_\lambda=\begin{cases}\bigg\lceil \frac{-\log \lambda- \log 2}{\log 3}\bigg\rceil &\text{if $\lambda< \frac 12$}\\0 &\text{if $\lambda\geq \frac 12$.}	
\end{cases}
\end{equation}
 \end{proposition}
\proof For $\lambda\geq \frac 12$, any element of $U(1)_\lambda=(\R/\Z,d)_\lambda$ is at distance $\leq \lambda$ from $0$, thus one can take $F=\emptyset$ as generating set since, by convention, $\sum_\emptyset=0$. Thus $\dim_{\spm}U(1)_\lambda=0$. Next, we assume $\lambda<\frac 12$. Let $F\subset U(1)$ be a generating set and let $k=\# F$. One easily sees that there are at most $3^k$ elements of the form $\sum_F \alpha_j j$. The subsets  $\{x\in U(1) \mid d(x,\alpha_j j)\leq \lambda\}$  cover $U(1)$, and since each of them has measure $2\lambda$ one gets the inequality $2 \lambda\cdot  3^k\geq 1$. Thus  $k\geq \frac{-\log \lambda- \log 2}{\log 3}$. For $\lambda=\frac 16$ one has $k\geq 1$ and the subset  $F=\{\frac 13\}$ generates, thus $\dim_{\spm}U(1)_\lambda=1$.  When $\frac{-\log \lambda- \log 2}{\log 3}=m$ is an integer, one has $\lambda=\frac 12 3^{-m}$. Let $F(m)=\{\frac 13,\ldots,\frac{1}{3^{m}}\}$. The minimal distance between two elements of $F(m)$ is the distance between $3^{-m}$ and $3^{-m+1}= 3\cdot  3^{-m}$ which is $2 \cdot 3^{-m}=4 \lambda$.
Let us show that $F(m)$ is a generating set. By Lemma \ref{bijectivecase} any integer $q$ in the interval $[-N,N]$, for $N=\frac 12 (3^m-1)$ can be written as  $q=\sum_{i=0}^{m-1} \alpha_i 3^i$, with $\alpha_i\in\{-1,0,1\}$. One then gets 
\[
q \cdot 3^{-m}=\sum_{i=0}^{m-1} \alpha_i 3^{i-m}=\sum_{j=1}^{m}\alpha_{m-j}3^{-j}.
\]
 Let $x\in [-\frac 12,\frac 12]$, then $3^m  x\in [-\frac {3^m}{2},\frac {3^m}{2}]$ and there exists an integer $q\in [-N,N]$ such that $ \vert 3^m  x-q\vert \leq \frac 12$. 
 Hence $d(x,q \cdot 3^{-m})\leq \lambda$. This proves that $F(m)$ is a generating set (see Definition \ref{generator}) and one derives $\dim_{\spm}U(1)_\lambda=m$. Assume now that $\frac{-\log \lambda- \log 2}{\log 3}\in(m,m+1)$, where $m$ is an integer. For any generating set $F$ of cardinality $k$ one has $k\geq \frac{-\log \lambda- \log 2}{\log 3}>m$ so that $k\geq m+1$. The subset $F(m+1)=\{\frac 13,\ldots,\frac{1}{3^{m+1}}\}$ fulfills the first condition of Definition \ref{generator} since the minimal distance between two elements of  $F(m+1)$ is $2 \cdot 3^{-m-1}$ which is larger than $\lambda< \frac 12 3^{-m}$. As shown above, the subset $F(m+1)$ is generating for $\lambda=\frac 12 3^{-m-1}$ and a fortiori for $\lambda>\frac 12 3^{-m-1}$ (as by assumption  
$\frac{-\log \lambda- \log 2}{\log 3}<m+1$). Thus one obtains $\dim_{\spm}U(1)_\lambda=m+1$ and  \eqref{dimh1} is proven.\endproof 
\begin{remark}\label{triadic}
The proof of Proposition \ref{caseh1} relies on  a definite advantage of the triadic expansion of numbers \cite{triadic}:  truncating a number is identical to rounding it. This property does not hold in the decimal system where rounding a number requires the knowledge of the next digit.
\end{remark}

We extend the ceiling function to negative values of the variable as follows
 \begin{equation}\label{ceilingprime}
\lceil x \rceil'	=\begin{cases} \lceil x \rceil\ \text{if}\ x\geq 0\\ 	-\lceil -x \rceil\ \text{if}\ x\leq  0.\end{cases}
\end{equation}
 
 We can now state and prove the absolute Riemann-Roch theorem for $\spzb$  over $\spm$. With $L$ the exceptional set defined in \eqref{excL} one has
 
 \begin{theorem}\label{rrspzb} 
Let $D$ be an Arakelov divisor  on $\spzb$. Then 
\begin{equation}\label{rrforq1}
\dim_{\spm}H^0(D)-\dim_{\spm}H^1(D)=\bigg\lceil \frac{\deg D+ \log 2}{\log 3}\bigg\rceil'	-{\mathbf 1}_L.
\end{equation}
\end{theorem}
 
 \proof  By appealing to the invariance under linear equivalence (Proposition \ref{setup1} $(iii)$), one may assume that $D=a\{\infty\}$. Then it follows from Proposition \ref{thedimh1} $(ii)$ that $\dim_{\spm}H^1(D)=\dim_{\spm}U(1)_\lambda$ for $\lambda=\exp(\deg(D))$. Assume first that $\exp(\deg(D))\geq \frac 12$, then by Proposition \ref{caseh1}: $\dim_{\spm}H^1(D)=0$, thus for $\deg(D)\geq -\log 2$ \eqref{rrforq1} follows from \eqref{rrforq}. Let us now assume that $\deg(D)< -\log 2$. Then $\dim_{\spm}H^0(D)=0$ since $H^0(D)=\{\ast\}$  when $\deg D<0$. Moreover  $\deg D\notin L$ since $L$ is lower bounded by $-\log 2$. Thus \eqref{rrforq1} follows from \eqref{dimh1} which gives, using \eqref{ceilingprime}
 \[
 -\dim_{\spm}H^1(D)=-\dim_{\spm}U(1)_\lambda=-\bigg\lceil \frac{-\log \lambda- \log 2}{\log 3}\bigg\rceil=\bigg\lceil \frac{\deg D+ \log 2}{\log 3}\bigg\rceil'.	
 \]
 This ends the proof of \eqref{rrforq1}. \endproof
 
\section{Duality}\label{sec5}

In this section we prove an absolute analogue of  Serre's duality, namely the following isomorphism of $\spm$-modules, for any divisor $D$ on $\spzb$: 
\[
H^0(D)\simeq \uhom_{\Gamma\cT_*}(H^1(K-D),U(1)_{\frac 14}).
\]
Here, the divisor $K=-2\{2\}$ plays the role of the canonical divisor. The choice of the  tolerant $\spm$-module $(U(1),d)_{\frac 14}$ as dualizing module is motivated by  Pontryagin duality (see \ref{sectpontdual}). One has $\dim_{\spm}(U(1)_{\frac 14})=1$. This equality, in fact,  also holds for the tolerant $\spm$-module $U(1)_{\lambda} $ for $\frac 16\leq \lambda<\frac 12$ :  the specific choice $\lambda=\frac 14$ is dictated by the invariance of the Riemann-Roch formula \eqref{rrforq1} when one switches $H^0$ and $H^1$ and replaces $D$ by $K-D$ (ignoring the exceptional set $L$).

 \subsection{$\uhom_\sss(\Vert H\R\Vert_\lambda,\Vert H\R\Vert_\mu)$}
 
 We start with the following general statement.  
 \begin{lemma}\label{homhom} Let $\lambda, \mu\in\R_{>0}$. The $\sss$-algebra structure of $H\R$ induces an isomorphism of $\sss$-modules
  	\begin{equation}\label{homhom1}
\Vert H\R\Vert_{\mu/\lambda}\simeq \uhom_\sss(\Vert H\R\Vert_\lambda,\Vert H\R\Vert_\mu).
\end{equation}
  \end{lemma}
\proof One starts by defining the morphism of $\sss$-modules 
\begin{equation}\label{etamap}
 	\eta: \Vert H\R\Vert_{\mu/\lambda}\to \uhom_\sss(\Vert H\R\Vert_\lambda,\Vert H\R\Vert_\mu).
 \end{equation}
Precisely, the multiplication in the $\sss$-algebra  $H\R$  determines  natural maps 
\[
m:H\R(X)\wedge H\R(Y)\to H\R(X\wedge Y)
\]
inducing, for a fixed pointed set $X$, the map $m_X: H\R(X)\to \Hom_\sss( H\R,H\R( X\wedge -))$.
The morphism  $\eta$ is defined as the natural transformation of functors which associates to  $X=k_+$, the map  $\eta_X:\Vert H\R\Vert_{\mu/\lambda}(X)\to \Hom_\sss(\Vert H\R\Vert_\lambda,\Vert H\R\Vert_\mu( X\wedge -))$ defined as the restriction of $m_X$. This restriction is meaningful in view of  \cite{CCprel} (Proposition 6.1).\newline
Next, we show that $\eta$ is an isomorphism. First, we determine the $\sss$-module $\Hom_\sss(\Vert H\R\Vert_\lambda,\Vert H\R\Vert_\mu)$. For integers $1\leq j\leq k$ we let as in \eqref{partmaps}
\[
\delta(j,k)\in \Hom_\gop(k_+,1_+), \quad \delta(j,k)(\ell):=\begin{cases} 1 &\text{if $\ell=j$}\\ * &\text{if $\ell\neq j$}.	
\end{cases}
\]
 Let $\phi\in \Hom_\sss(\Vert H\R\Vert_\lambda,\Vert H\R\Vert_\mu)$. By construction, the natural transformation $\phi$ reads, for each $k$, as a map $\phi(k_+):\Vert H\R\Vert_\lambda(k_+)\to \Vert H\R\Vert_\mu(k_+)$ and by naturality we have
 \begin{equation}\label{etamap1}
\Vert H\R\Vert_\mu(\delta(j,k))\circ \phi(k_+)=\phi(1_+)\circ \Vert H\R\Vert_\lambda(\delta(j,k)).
\end{equation}
 Since an element  $y\in\Vert H\R\Vert_\mu(k_+)$ is determined by its components $y_j\in \R$, $1\leq j\leq k$,
with $y_j=\Vert H\R\Vert_\mu(\delta(j,k))(y)$, \eqref{etamap1} shows that  $\phi(k_+)$ is uniquely determined by the map  $\phi(1_+)$ acting componentwise, \ie 
\begin{equation}\label{etamap2}
	\phi(k_+)((x_j))=\left(\phi(1_+)(x_j)\right).
\end{equation}
Moreover, the map $f=\phi(1_+):[-\lambda,\lambda]\to [-\mu, \mu]$  fulfills 
\begin{equation}\label{etamap3}
	f(x+y)=f(x)+f(y)\quad\forall x,y\  s. t. \ \vert x\vert +\vert y\vert \leq \lambda, 
\end{equation}
as one sees using the naturality of $\phi$ for the map $\sigma\in \Hom_\gop(2_+,1_+)$, \ie using 
\[
(\Vert H\R\Vert_\mu(\sigma)\circ \phi(2_+))(x,y)=(\phi(1_+)\circ \Vert H\R\Vert_\lambda(\sigma))(x,y)\Longrightarrow f(x)+f(y)=f(x+y).
\]
By \eqref{etamap3} one has $f(x 2^{-n})= 2^{-n} f(x)$, for any $x\in [-\lambda,\lambda]$ and $n\geq 0$. Thus the `germ of map' $f$ uniquely extends to a map $\tilde f:\R\to \R$ defined by $\tilde f(x):=2^n f(x 2^{-n})$ for any $n$ such that $\vert x 2^{-n}\vert \leq \lambda$. Moreover, again by \eqref{etamap3}, the map $\tilde f$ is additive and since $\tilde f([-2^{-n}\lambda,2^{-n}\lambda])\subset [-2^{-n}\mu,2^{-n}\mu]$, it is also continuous and hence determined by the multiplication by a real number $r$. One has $\vert r \lambda \vert \leq \mu$ and hence $r\in \Vert H\R\Vert_{\mu/\lambda}(1_+)$. Thus one gets $\phi(1_+)(x)=r x$, $\forall x \in [-\lambda,\lambda]$. By  \eqref{etamap2} one obtains $\phi(k_+)((x_j))=(r x_j)$, $\forall (x_j)\in \Vert H\R\Vert_\lambda(k_+)$. This shows that for $X=1_+$, $\eta_X:\Vert H\R\Vert_{\mu/\lambda}(X)\to \Hom_\sss(\Vert H\R\Vert_\lambda,\Vert H\R\Vert_\mu( X\wedge -))$ is bijective. The next step is to determine $\Hom_\sss(\Vert H\R\Vert_\lambda,\Vert H\R\Vert_\mu( X\wedge -))$ for $X=\ell_+$. Let $\phi \in \Hom_\sss(\Vert H\R\Vert_\lambda,\Vert H\R\Vert_\mu( \ell_+\wedge -))$. One has $\ell_+\wedge k_+=(\ell k)_+$ and an element 
$z\in \Vert H\R\Vert_\mu( \ell_+\wedge k_+)$ is determined by its components $z(i,j)\in \R$ for $1\leq i\leq \ell$, $1\leq j\leq k$ such that $\sum \vert z(i,j)\vert \leq \mu$. In particular, for each $j$, the $z(i,j)$, $1\leq i\leq \ell$, are the components of $z_j=\Vert H\R\Vert_\mu( \id\wedge \delta(j,k)))(z)\in \Vert H\R\Vert_\mu(\ell_+)$. This implies by applying the naturality of $\phi$,  that  
\begin{equation}\label{etamap4}
	\phi(k_+)(z)=\left(\phi(1_+)(z_j)\right).
\end{equation}
As shown above the map $f=\phi(1_+):[-\lambda,\lambda]\to \Vert H\R\Vert_\mu(\ell_+)$  fulfills 
\begin{equation}\label{etamap5}
	f(x+y)=f(x)+f(y)\quad\forall x,y\  s. t. \ \vert x\vert +\vert y\vert \leq \lambda 
\end{equation}
and it extends to an additive map $\tilde f: \R \to \R^\ell$ which is continuous since $\tilde f$ maps the interval $[-\lambda,\lambda]$ inside $\Vert H\R\Vert_\mu(\ell_+)$. Thus there exists 
real numbers  $r_i$, $1\leq i\leq \ell$, such that $\tilde f(x)=(r_i x)$ $\forall x\in \R$. One has $\sum \vert r_i \lambda  \vert \leq \mu$ and thus $r=(r_i)\in \Vert H\R\Vert_{\mu/\lambda}(\ell_+)$. Finally, using \eqref{etamap4} it follows that $\phi=\eta_X(r)$. This proves that  $\eta$ as in \eqref{etamap} is an isomorphism.\endproof

\subsection{Pontryagin duality}\label{sectpontdual}

In order to formulate Pontryagin duality in this context we consider, for $\lambda>0$,  the $\spm$-module  $U(1)_\lambda=(U(1),d)_\lambda$, where $U(1)$ is the abelian group $\R/\Z$ endowed with its canonical metric $d$ of length $1$. For a metric abelian group $(A,d)$, we denote by $\widehat A$  the abelian group of continuous characters, \ie of  continuous group homomorphisms $A\to U(1)$, where $A$ is endowed with the topology associated to the metric $d$. 
We retain the notations of section~\ref{secA}. Next statement is motivated by  Lemma~\ref{homhom}.

\begin{proposition}\label{pontrj} Let $(A,d)$ be an abelian group endowed with a translation invariant metric $d$. 
\begin{enumerate}
\item[(i)] For $\lambda, \mu\in\R_{>0}$, the $\spm$-module $\uhom_{\Gamma\cT_*}((A,d)_\lambda,U(1)_\mu)$ is isomorphic to the sub $\spm$-module of $H\widehat A$ which, on the set $k_+$, is given by k-tuples $(\chi_j)$, $1\leq j\leq k$ of continuous characters  
	$\chi_j\in \widehat A$ such that, with $\vert x\vert:=d(x,0)$, and for all finite collections $\{x_i\}\subset A$, fulfill
	\begin{equation}\label{inequ1}
	\sum_i \vert x_{i}\vert \leq \lambda~ \Rightarrow ~\sum_{i,j} \vert \chi_j(x_{i})\vert \leq \mu.
	\end{equation}
	\item[(ii)] Let $p:A'\to A$ be a surjective morphism of abelian groups and let $p^*(A,d)_\lambda$ be the pullback as in Proposition \ref{pullback}. One has the following canonical isomorphism 
	\begin{equation}\label{isoduals}
	\uhom_{\Gamma\cT_*}(p^*(A,d)_\lambda,U(1)_\mu)\simeq \uhom_{\Gamma\cT_*}((A,d)_\lambda,U(1)_\mu).
	\end{equation}
	\end{enumerate}
	\end{proposition}
	\proof $(i)$~Let $\phi\in\uhom_{\Gamma\cT_*}((A,d)_\lambda,U(1)_\mu)(1_+)=\Hom_{\Gamma\cT_*}((A,d)_\lambda,U(1)_\mu)$. By applying the forgetful functor $\cF:\cT\longrightarrow \Se$ which associates to $(X,\cR)$ the set $X$ (see Proposition~\ref{adjoint}), one obtains  an element $\cF(\phi)\in \Hom_\sss(HA,HU(1))$. Since the Eilenberg-MacLane functor   $H$ determines a full and faithful embedding of the category of abelian groups inside the category of $\sss$-modules,  there exists a unique  group homomorphism $\chi:A\to U(1)$  such that $H(\chi)=\cF(\phi)$. The condition that $\phi$ preserves the relation $\cR_k$  on the set $k_+$ means that for any $x_i,y_i\in A$, $1\leq i\leq k$ such that $\sum d(x_i,y_i)\leq \lambda$ one has $\sum d(\chi(x_i),\chi(y_i))\leq \mu$. By translation invariance of the metrics this condition is equivalent to 
	\begin{equation}\label{inequ2}
	\sum_i \vert x_{i}\vert \leq \lambda~ \Rightarrow ~\sum_{i} \vert \chi(x_{i})\vert \leq \mu.
	\end{equation}
	This shows that $\Hom_{\Gamma\cT_*}((A,d)_\lambda,U(1)_\mu)$ consists exactly of the group homomorphisms $\chi:A\to U(1)$ fulfilling \eqref{inequ1}. Specializing \eqref{inequ2} to the case where all $x_i=x$, $i=1,\ldots, n$,  one obtains the implication $\vert x\vert \leq \lambda/n \Rightarrow \vert \chi(x)\vert \leq \mu/n $, and hence that $\chi$ is uniformly continuous. Let $\phi\in\uhom_{\Gamma\cT_*}((A,d)_\lambda,U(1)_\mu)(k_+)=\Hom_{\Gamma\cT_*}((A,d)_\lambda,U(1)_\mu(k_+\wedge -))$. The object $U(1)_\mu(k_+\wedge -)$ of $\Gamma\cT_*$ is $(U(1)^{\times k},d_k)_\mu$, where the metric $d_k$ on the product group $(U(1)^{\times k}$ is defined by
	\[
	d_k((x_i),(y_i)):=\sum_1^kd(x_i,y_i)\quad\forall x_i,y_i \in U(1).
	\]
	Replacing $U(1)_\mu$ with $(U(1)^{\times k},d_k)_\mu$ in the first part of the proof one obtains that $\cF(\phi)\in \Hom_\sss(HA,HU(1)^{\times k})=H((\chi_j))$, where $(\chi_j)$, $1\leq j\leq k$ is a $k$-tuple of  characters of $A$ fulfilling \eqref{inequ1}. It follows that $\chi_j\in \widehat A$.\newline
	$(ii)$~Let $\phi\in\uhom_{\Gamma\cT_*}(p^*(A,d)_\lambda,U(1)_\mu)(1_+)=\Hom_{\Gamma\cT_*}(p^*(A,d)_\lambda,U(1)_\mu)$. As in the proof of $(i)$, there exists a group homomorphism $\chi': A'\to U(1)$ such that $H(\chi')=\cF(\phi)$. Moreover $\chi'$ preserves the relation $\cR_k$ for any $k$, and this implies 
	\[
	 \sum_{i=1}^k \vert p(x'_i)\vert \leq \lambda ~\Rightarrow~ \sum_{i=1}^k \vert \chi'(x'_i)\vert \leq \mu, \qquad \forall (x'_i)\in (A')^{\times k}.
	\]
	In particular, taking all $x'_i =x'\in \ker(p)$ one obtains $\vert \chi'(x')\vert\leq \mu/k $ $\forall k$, and hence 
	$\chi'(\ker(p))=\{1\}$. This implies that there exists a group homomorphism $\chi:A\to U(1)$ such that $\chi'=\chi\circ p$. 
	\endproof 
	
We can now state and prove Serre's duality.

\begin{theorem}\label{serredual} Let $D=\sum_j a_j\{p_j\} + a\{\infty\}$ be an Arakelov divisor  on $\spzb$.  There is a canonical isomorphism of $\spm$-modules 
\begin{equation}\label{serredual1}
H^0(K-D)\simeq \uhom_{\Gamma\cT_*}(H^1(D),U(1)_{\frac 14}),
\end{equation}
where $K$ is the divisor   $K=-2\{2\}$.
\end{theorem}
\proof By Proposition \ref{thedimh1},  with $\lambda=\exp a$, one has $H^1(D)=\pi^*((\R/L,d)_\lambda)$ and by  Proposition \ref{pontrj}, $(ii)$, one gets the isomorphism 
\[
\uhom_{\Gamma\cT_*}(H^1(D),U(1)_{\frac 14})\simeq \uhom_{\Gamma\cT_*}((\R/L,d)_\lambda,U(1)_{\frac 14}).
\]
In fact, we can assume that $D=a \{\infty\}$, with $a=\deg D$ so that $L=\Z$. Then we apply Proposition \ref{pontrj} $(i)$, with $A=\R/\Z$ and $\mu=\frac 14$. One has $\widehat A=\Z$ and the characters $\chi_n\in \widehat A$ are given by multiplication by $n$, \ie $\chi_n(s):=ns\in \R/\Z$, $\forall s\in \R/\Z$. Next, we need to determine the $\spm$-submodule of $H\widehat A=H\Z$ which, on the set $k_+$, is given by $k$-tuples $(n_j)$, $1\leq j\leq k$ of  characters  
	$n_j\in \Z$ such that \eqref{inequ1} holds. This means, using the distance $d$ on $\R/\Z$, that
	\begin{equation}\label{inequ1for}
	\sum_i d (x_{i},0) \leq \lambda~ \Rightarrow ~\sum_{i,j} d( n_j x_{i},0) \leq \frac 14.
	\end{equation}
	The distance $d(x,0)$ is given, for any $x'\in \R$ in the class of $x$ by the distance between $x'$ and the closed subset $\Z\subset \R$. Thus for any integer $n$ one has: $d(nx,0)\leq \vert n\vert d(x,0)$. Assume that $\sum_j \vert n_j\vert \leq \frac{1}{4\lambda}$, then  \eqref{inequ1for} follows since
	\[
	\sum_{i,j} d( n_j x_{i},0)\leq \sum_{i,j} \vert n_j\vert d( x_{i},0)\leq \frac{1}{4\lambda}\sum_i d (x_{i},0).
	\]
	Conversely, assume  \eqref{inequ1for}. Then repeating $m$ times the same $x$,  gives  
	\[
	 d (x,0) \leq \frac\lambda m ~\Rightarrow ~\sum_{j} d( n_j x,0) \leq \frac {1}{4m}.
	\]
 Taking $m$ large enough and  $x=\frac\lambda m$ one obtains $\sum_{j} \vert n_j \vert x\leq \frac {1}{4m}$, and hence $\sum_j \vert n_j\vert \leq \frac{1}{4\lambda}$. This proves that the $\spm$-submodule of $H\widehat A=H\Z$ determined by  \eqref{inequ1} is equal to $\Vert H\Z\Vert_{\frac{1}{4\lambda}}$ which gives \eqref{serredual1}.\endproof

\begin{appendix}

\section{Tolerance $\sss$-modules}\label{secA}

 The construction of the category $\Gamma\cS_*$ of $\Gamma$-spaces (see appendix~\ref{secB}) can be broadly generalized by considering in place of  the category $\cS_*$ of simplicial pointed sets any  pointed category $\cC$ with initial and final object $*$. In this way, one obtains a category $\Gamma\cC$ of pointed covariant functor $\gop \longrightarrow \cC$. We shall apply this formal construction to  the category of tolerance relations and introduce the notion of tolerant $\spm$-modules which plays a central role in the development  of the absolute Riemann-Roch problem.
We start with the following general fact
 
 \begin{lemma}\label{catcat} Let $\cC$ be a pointed category with initial and final object $*$. Then $\Gamma\cC$ is naturally enriched in $\sss$-modules. More  precisely, the following formula endows the internal $\uhom_{\Gamma\cC}(A,B)$ with a structure of $\sss$-module defined by
 \begin{equation}	\label{catcat1}
 	\uhom_{\Gamma\cC}(A,B)(k_+):=\Hom_{\Gamma\cC}(A,B(k_+\wedge -))\qquad k\in\N.
 \end{equation}
 \end{lemma}
\proof Let $\phi\in \Hom_\gop(k_+,\ell_+)$. For every object $F$ of $\gop$ the morphism $\phi \wedge \id\in \Hom_\gop(k_+\wedge F,\ell_+\wedge F)$ gives, by functoriality of $B:\gop \longrightarrow\cC$, a morphism $B(\phi \wedge \id)\in \Hom_\cC(B(k_+\wedge F),B(\ell_+\wedge F))$. These morphisms define a natural transformation of functors $B(\phi\wedge -)\in \Hom_{\Gamma\cC}(B(k_+\wedge -), B(\ell_+\wedge -))$, and one obtains the functoriality on the right hand side of \eqref{catcat1} using the left composition 
\[
\psi\in \Hom_{\Gamma\cC}(A,B(k_+\wedge -))\mapsto B(\phi\wedge -)\circ \psi \in 
\Hom_{\Gamma\cC}(A,B(\ell_+\wedge -)).
\]
 \endproof 
 
  \subsection{The category $\Gamma\cT_*$}\label{gamma-t}

 A tolerance relation $\cR$ on a set $X$ is a reflexive and symmetric relation on  $X$.  Equivalently, $\cR$ is a subset $\cR\subset X\times X$ which is symmetric and containing the diagonal. We shall denote by $\cT$ the category of tolerance relations $(X,\cR)$. Morphisms in  $\cT$  are defined by \[ 
\Hom_\cT((X,\cR),(X',\cR')):=\{	\phi: X \to X', \  \phi(\cR)\subset \cR'\}.
\]
We denote  $\cT_*$ the pointed category under the object $\{\ast\}$ endowed with the trivial relation. One has the following

\begin{definition} \label{cT} A tolerant $\sss$-module is  a pointed covariant functor $\gop \longrightarrow \cT_*$.
We denote by $\Gamma\cT_*$ the category of tolerant $\sss$-modules.
\end{definition}

Next statement is an easy but useful fact

\begin{proposition}\label{adjoint} 
\begin{enumerate}
\item[(i)] The functor $\Se\longrightarrow \cT$ which endows a set with the diagonal relation, embeds the category of sets as a full subcategory of $\cT$, and consequently the category of $\sss$-modules  as a full subcategory of the category $\Gamma\cT_*$.
\item[(ii)] The forgetful functor is the right adjoint of the inclusion in $(i)$.	
\end{enumerate}
\end{proposition}

\subsection{The tolerant $\sss$-module $(A,d)_\lambda$}

 A relevant example of tolerant $\sss$-module is given by the following construction. Let $A$ be an additive abelian group. A translation invariant metric $d$ on $A$ is a metric on $A$ that satisfies $d(x,y)=d(x-y,0)$, so that the triangle inequality can be  read as $d(x+y,0)\leq d(x,0)+d(y,0)$ $\forall x,y \in A$. This fact implies that the inequality 
\begin{equation}\label{inequ}
d\bigg(\sum_I x_i,\sum_I y_i\bigg)\leq \sum_I d(x_i,y_i)
\end{equation}
holds for any finite index set $I$ and maps $x, y:I\to A$.

\begin{proposition}\label{defnA} Let $(A,d)$ be an abelian group endowed with a translation invariant metric $d$ and let $\lambda\in \R_{>0}$. The following relations   turn the Eilenberg?MacLane $\sss$-module $HA$ into a tolerant $\spm$-module $(A,d)_\lambda: \gop\longrightarrow \cT_*$:
\begin{equation}\label{inequdef} \mathcal R_k= \{((x_i),(y_i))\in A^{\times k}\times A^{\times k}\mid \sum_{I=k_+} d(x_i,y_i)\leq \lambda \}\qquad \forall k_+,~k\in\N.
\end{equation}
\end{proposition}
\proof For any $\phi\in \Hom_\gop(k_+,\ell_+)$, the map $HA(\phi):HA(k_+)\to HA(\ell_+)$ fulfills $HA(\phi)^{\times 2}(\cR_k)\subset \cR_\ell$. Indeed, this follows from \eqref{inequ} applied to the finite sets $I_j=\{i\in k_+\mid \phi(i)=j\}$ which label pairs of elements of $A$. \endproof 

 Let  $U(1)$ be the abelian group $\R/\Z$ endowed with the canonical metric $d$ of length $1$. We shall denote by $U(1)_\lambda$ the tolerant $\sss$-module $(U(1),d)_\lambda$,  ($\lambda\in\R_{>0}$).

 \subsection{The dimension of a tolerant $\spm$-module}
 
 In this part we  introduce a notion of dimension for a tolerant $\spm$-module that naturally generalizes, in the absolute context, the  definition of dimension of a vector space. 
 
\begin{definition}\label{generator}
 Let $(E,\cR)$ be a tolerant $\spm$-module. 
	A subset $F\subset E(1_+)$ generates $E(1_+)$ if the following two conditions hold
	\begin{enumerate}
	\item For $x,y\in F$, with $x\neq y ~\Longrightarrow~ (x,y)\notin \cR$	
	\item For every   $x\in E(1_+)$ there exists  $\alpha_j\in \{-1,0,1\}$,  $j\in F$ and $y\in E(1_+)$ such that $y=\sum_F \alpha_j j\in E(1_+)$ in the sense of \eqref{defnsum},  and $(x,y)\in \cR$.
	\end{enumerate}
  The dimension $\dim_{\spm}(E,\cR)$ is defined as the minimal cardinality of a generating set $F$.
	\end{definition}

 To familiarize with this notion we  prove the following 
 
 \begin{proposition}\label{pullback} Let $p:A\to B$ be a  morphism of abelian groups and $(HB,\cR)$ a tolerant $\spm$-module.
 \begin{enumerate}
\item[(i)] Consider the relation  $p^*(\cR_k):=\{(x,y)\in HA(k_+)\times HA(k_+)\mid (Hp(x),Hp(y))\in \cR_k\}$. Then the pair $p^*(HB,\cR):=(HA,p^*(\cR))$ is a tolerant $\spm$ module.
\item[(ii)] If $p$ is surjective: $\dim_{\spm}(p^*(HB,\cR))=\dim_{\spm}(HB,\cR)$.
\end{enumerate}
 \end{proposition}
 \proof $(i)$~Since $(HB,\cR)$ is a tolerant $\spm$-module, for any $\phi\in \Hom_\gop(k_+,\ell_+)$ one has $HB(\phi)^{\times 2}(\cR_k)\subset \cR_\ell$.  Also  
 \[
 HA(\phi)^{\times 2}(p^*\cR_k)\subset p^*\cR_\ell \qquad\forall \phi\in \Hom_\gop(k_+,\ell_+),
 \]
 which shows that $p^*(HB,\cR)=(HA,p^*(\cR))$ is a tolerant $\spm$ module. \newline 
 $(ii)$~Let $F\subset HB(1_+)=B$ be a generating set for $(HB,\cR)$, and let $F'\subset HA(1_+)=A$ be a lift of $F$, with $\# F'=\# F$. Let us show that $F'$ is a generating set for $p^*(HB,\cR)=(HA,p^*(\cR))$. Let $a\in HA(1_+)=A$, then there exists coefficients $\alpha_j\in \{-1,0,1\}$,  $j\in F$, such that  $(p(a),\sum_F \alpha_j j)\in \cR$. It follows that using the lifts $j'\in F'$ of $j\in F$ one has $(a,\sum_{F'} \alpha_j j')\in p^*(\cR)$, hence $F'$ is a generating set for $p^*(HB,\cR)$. Thus $\dim_{\spm}(p^*(HB,\cR))\leq \dim_{\spm}(HB,\cR)$. Conversely, let $F'\subset A$ be a generating set for $p^*(HB,\cR)$ and $F:=p(F')$. Then condition 1. of Definition \ref{generator} for $F'$ implies the same condition for $F$, thus one has $\# F'=\# F$. Let  $b\in HB(1_+)=B$ and $a\in A$ with $p(a)=b$. Then there exists coefficients $\alpha_j\in \{-1,0,1\}$,  $j\in F'$, such that  $(a,\sum_{F'} \alpha_j j')\in p^*(\cR)$. This implies
 $(b,\sum \alpha_j j)\in \cR$ so that $F$ is a generating set for $(HB,\cR)$. \endproof  
 
 Next, we apply this functorial machinery to the geometry of $\spzb$. We retain the notations of section~\ref{sec2}.
 
 \begin{proposition}\label{thedimh1} Let $D=\sum_j a_j\{p_j\} + a\{\infty\}$ be an Arakelov divisor on $\spzb$ and let $\pi:\A_\Q\to \R/L$ be the projection of the  adeles on their archimedean component modulo the lattice
 \begin{equation}\label{weil6}
 L= H^0(\Spec\Z,\cO(D)_f):=\{q \in \Q\mid \vert q\vert_\nu \leq \exp(D)(\nu),~\forall \nu\neq \infty \}.
\end{equation}
\begin{enumerate}
\item[(i)] Let $d$ be the metric on $\R/L$ induced by the standard metric on $\R$ and set $\lambda=\exp a$. Then one has $H^1(D)=\pi^*((\R/L,d)_\lambda)$. 
\item[(ii)] 
 	$\dim_{\spm}H^1(D)=\dim_{\spm}(\R/L,d)_\lambda$. 	
 	\end{enumerate}
 \end{proposition}
\proof 
$(i)$~Let $j: \A_\Q^f\to \A_\Q$, $j(x):= (x,0)$, be the embedding of  finite adeles in adeles. Using the ultrametric property of  the local norms at  the finite places one sees that $j(\cO(D)_f)=\cO(D)\cap j(\A_\Q^f)\subset \A_\Q$ is a compact subgroup. Set $G=\Q\times j(\cO(D)_f)$: one has $j(\cO(D)_f)\cap\Q=\{0\}$ since all the adeles in $j(\cO(D)_f)$ have archimedean component equal to $0$. Thus, the  restriction  of the morphism of  addition,  $\alpha:\Q\times \A_\Q \to \A_\Q$, $\alpha(q,a)=q+a$,  to $G$   determines an isomorphism of $G$   with  the  subgroup  $\alpha(G)=\Q+j(\cO(D)_f)$ of $\A_\Q$. Note, in particular, that $\alpha(G)$  is closed  in $\A_\Q$, since $\Q$ is discrete  (hence closed) and $j(\cO(D)_f)$ is compact. In the following, we identify (set-theoretically)  $\A_\Q$ with the product $\A^f_\Q\times \R$  endowed with the two projection morphisms   $p_f: \A_\Q\to \A^f_\Q$ and $p_\infty: \A_\Q\to \R$. The subgroup $p_f(\Q)\subset \A^f_\Q$ is dense and the subgroup $p_f(j(\cO(D)_f))=\cO(D)_f\subset \A^f_\Q$ is open.  Hence   $p_f(\alpha(G))=\A^f_\Q$. Thus  the projection $p_\infty$ induces   the isomorphism  of groups   $p:\A_\Q/\alpha(G)\stackrel{\sim}{\to} \R/L$, where $L=p_\infty(\ker(p_f\circ \alpha)) $. The kernel  of the composite $p_f\circ \alpha: G\to \A_\Q^f$  is the group of pairs $(q,a)\in \Q\times j(\cO(D)_f)$ such that $p_f(q)+p_f(a)=0$. Such pairs are determined by the value of $q$ and thus 
\begin{equation}\label{weil6}
L=p_\infty(H^0(\Spec\Z,\cO(D)_f))=\{q \in \Q\mid \vert q\vert_\nu \leq \exp(D)(\nu), \quad\forall \nu\neq \infty \}.
\end{equation}
By Proposition \ref{setup1}, $H^1(D)$ is the tolerant  $\spm$-module $H^1(D) = (H\A_\Q, \cR)$ where the relations are given by \eqref{cokerpsi}, \ie
\[
 \cR_k:=\{ (x,y)\in H\A_\Q(k_+)\times  H\A_\Q(k_+)\mid x-y\in {\rm Image}\, \psi(k_+)\},
\]
and where $\psi$ is as in \eqref{psi}.
    By construction  one has $\Q\times \cO(D)\simeq G\times [-e^a,e^a]$.
  After quotienting both sides of \eqref{psi} by $G=\Q\times j(\cO(D)_f)$, the map $\psi$  becomes
 	 \begin{equation}\label{weil5}
\phi: \Vert H\R \Vert_{e^a}\to H(\R/L), \quad  \phi(x)=x+L \in \R/L\quad\forall x\in [-e^a,e^a]\subset \R,
\end{equation}
where $L\subset \Q\subset \R$ is the lattice \eqref{weil6}. Thus one obtains that the relation $\cR$ is equal to the inverse image by the map $\pi:\A_\Q\to \R/L$ of the relation \eqref{inequdef}.\newline
$(ii)$~Follows from Proposition \ref{pullback} $(ii)$.\endproof

\section{The $\Gamma$-space $\bfh(D)$}\label{secB}

Let $\phi: A \to B$ be a morphism of abelian groups. To $\phi$ one associates  the following (short) complex  $\mathfrak C=\{C_n,\phi_n\}$ of abelian groups indexed in  non-negative degrees
\begin{equation}\label{short}
\mathfrak C=\{C_1\stackrel{\phi_1}{\to} C_0\};\qquad C_0=B, \ C_1=A, \quad \phi_1(x):=\phi(x).\end{equation}
 
  The Dold-Kan correspondence associates to $\mathfrak C$ the  simplicial abelian group (see \cite{GJ} III.2, Proposition 2.2) 
  \begin{equation}\label{descriptshorto}
\cA_n=\bigoplus_{\cF(n,k)} C_k, \quad \cF(n,k):=\{\sigma\in \Hom_\Delta([n],[k])\mid \sigma([n])=[k]\}
\end{equation}
where $\Delta$ denotes the simplicial category.
The direct sum in \eqref{descriptshorto} repeats the term $C_k$ of the chain complex as many times as  the number of elements of  the set $\cF(n,k)$ of surjective morphisms $\sigma\in \Hom_\Delta([n],[k])$. For the short complex  $\mathfrak C$, the allowed values of $k$ are $k=0,1$. Therefore the set $\cF(n,0)$ is  reduced to   $\Hom_\Delta([n],[0])=:\Delta_n^0$, \ie to a single point.  A morphism  $\xi\in\Hom_\Delta([n],[1])=:\Delta^1_n$ is characterized by  $\xi^{-1}(\{1\})$ which is an hereditary subset of $[n]$. It follows that the  vertices in $\Delta^1_n$  are   labelled by the
\begin{equation}\label{identxi}
\xi_j, ~ 0\leq j\leq n+1, \quad \xi_j(k)=1\iff k\geq  j.
\end{equation}
 For each integer $n\geq 0$  the finite set $F(n):=\cF(n,1)$ of surjective elements  $\xi\in\Delta^1_n$  excludes $\xi_{0}$ and  $\xi_{n+1}$, thus the set $F(n)=\{\xi_j, \, 1\leq j\leq n \}$ has $n$ elements. This gives the  identification  
 \begin{equation}
\cA_n= HB(1_+)\oplus HA(F(n)_+)\label{descriptshort}
\end{equation}  
We refer  to \cite{GJ} (III.2, pp. 160-161) for a detailed description of  the simplicial structure,   namely the definition  for each $\theta \in
 \Hom_\Delta([m],[n])$ of a map of sets $\cA(\theta):\cA_n\to \cA_m$. 
 Next, we introduce some notations.
 
  We  identify the  opposite  $\dop$  of the simplicial category with (the skeleton of) the category of finite intervals.   An interval  is  a totally ordered set with the smallest element distinct from the largest one.  A morphism of intervals is a non-decreasing map  that preserves the smallest and largest elements. The canonical contravariant functor  $\Delta\longrightarrow \dop$   which identifies the opposite category of  $\Delta$  with  $\dop$ (described by intervals as above),  maps the  finite ordinal object $[n]=\{0,1,\ldots,n\}$ in $\Delta$  to the interval $[n]^*=\{0, \ldots, n+1\}$. 

We denote by $\crel$    the category of pairs  of pointed sets $(X,Y)$, with $X\supset Y$. The morphisms are maps  of pairs of pointed sets. We let $c:\crel\longrightarrow \Ses$ be the functor $(X,Y)\to X/Y$  of collapsing  $Y$ to the base point $\ast$.

Let  $\partial:\dop\longrightarrow \crel $ be the functor that  replaces an interval $I$ with the pair $\partial I=(X,Y)$, where $X$ is the set $I$ pointed by its smallest element, and $Y\subset X$ is the subset formed by the smallest and largest elements of $I$.

Finally, we denote by $\gamma:\crel \longrightarrow\Gamma\crel$ (see section \ref{secA}) the functor that associates to an object $(X,Y)$ of $\crel$ the covariant functor
\[
\gamma(X,Y):\gop\longrightarrow \crel,\quad k_+:=\{0,\ldots,k\}\mapsto (X\wedge k_+,Y\wedge k_+).
\]
\vspace{.01in}

  The following formula defines a covariant functor that associates to a pair of pointed sets  an abelian group directly related to the morphism $\phi: A\to B$
 \begin{equation}\label{descript1}
  H_\phi:\crel\longrightarrow \Ab,  \qquad	H_\phi(X,Y):=HB(Y)\times HA(X/Y).
 \end{equation}
 On morphisms	$f:(X,Y)\to (X',Y')$  in $\crel$ with $\alpha= (\alpha_Y,\alpha_{X/Y})\in H_\phi(X,Y)$,  the functor $H_\phi$ acts as follows
 \begin{align}\label{descript2}
 	H_\phi(f)(\alpha)&=\left(\alpha_{Y'}, HA(f)(\alpha_{X/Y})\right) \\\notag  \alpha_{Y'}(y')&:=HB(f)(\alpha_{Y})(y')+\sum_{x\in X\setminus Y, f(x)=y'}\phi(\alpha_{X/Y}(x)),\quad y'\neq *.\notag
 \end{align}
   By \cite{CCAtiyah} (Proposition~4.5), the Dold-Kan correspondence for the short complex \eqref{short}, \ie  the simplicial abelian group $\mathcal A_n$  in  \eqref{descriptshort}   is canonically isomorphic    to the composite functor $H_\phi\circ \partial: \dop\longrightarrow \Ab$, with $H_\phi$ as in \eqref{descript1}. 
  By composing $H_\phi$  with the Eilenberg-MacLane functor $H$ one obtains  a covariant functor $HH_\phi=H\circ H_\phi :\crel\longrightarrow \smod$  to the category of $\sss$-modules,  which is naturally isomorphic  to the functor $((U\circ H_\phi)\times \id)\circ \gamma$, where $U:\Ab\longrightarrow \Ses$ is the forgetful functor (see \opcit Lemma~4.6). Moreover, as shown in \opcit (Theorem~4.7)  the $\Gamma$-space associated by the Dold-Kan correspondence to the complex $\mathfrak C$ is canonically isomorphic  to the functor 
 \begin{equation}\label{gammasp}((U\circ H_\phi)\times \id)\circ \gamma\circ \partial : \dop \longrightarrow \smod.
\end{equation}
  Notice  that the above construction  involves the morphism $\phi:A\to B$  {\em only} through  the composite functor $U\circ H_\phi:\crel \longrightarrow \Ses$. This latter functor  is still meaningful  when one restricts $\phi$ to a sub-$\sss$-module  $E$    of the $\sss$-module $HA$ and it is given by 
  \begin{equation}\label{descript1sub}
   H_\phi\vert_E:\crel\longrightarrow \Ses,  \qquad	H_\phi(X,Y):=HB(Y)\times E(X/Y).
 \end{equation} 
 This  provides the following construction 
 
 \begin{definition}  Let $\phi:A\to B$ be  a morphism of abelian groups and $E$ a sub $\sss$-module of the $\sss$-module $HA$. We denote by $ \Gamma(\phi\vert_E)$ 
 	the $\Gamma$-space obtained as a sub-functor of \eqref{gammasp}
 \begin{equation}\label{descript1sub1}
 \Gamma(\phi\vert_E):= (H_\phi\vert_E\times \id)\circ \gamma\circ \partial : \dop \longrightarrow \smod.
 \end{equation}
 \end{definition}
 
  When evaluated on the object $1_+=\{0,1\}$ of $\gop$, the $\Gamma$-space $\Gamma(\phi\vert_E)$  defines   a sub-simplicial set of the Kan simplicial set $\mathcal A_n$  in \eqref{descriptshort}.  However it is not in general a Kan simplicial set, thus one needs to exert care  when  considering its homotopy.
 We refer to \cite{CCgromov}  2.1,  for the definition of the homotopies  used there. The relation $\cR$ of homotopy between $n$-simplices $(x,y)\in X_n\times X_n$  is defined as follows
\begin{equation}\label{Maydefn}
 (x,y)\in \cR \iff \partial_j x=\partial_jy\, \forall j \, \&\,  \exists z~\text{s.t.}~ \partial_jz=s_{n-1}\partial_j x\, \forall j<n, \ \partial_nz=x,\, \partial_{n+1}z=y.
  \end{equation}

In general, the  relation $\cR$   in  \eqref{Maydefn} fails to be an equivalence relation.  In place of the quotient we consider pairs of sets and relations.  We define $\pik_n$ to be the set  of spherical elements in $X_n$ (\ie of $n$-simplices $x$, with $\partial_j x=*$ $\forall j$), endowed with the relation $\cR$. Then we have the following result (we refer to section~\ref{secA} for the notion of tolerant module)

\begin{proposition}\label{proppi0} Let $\phi:A\to B$ be  a morphism of abelian groups and let $E$ be a sub-$\spm$-module of the Eilenberg-MacLane $\spm$-module $HA$ (where $(\pm 1)x:=\pm x$).
\begin{enumerate}
\item[(i)] The homotopy $\pik_1(\Gamma(\phi\vert_E))$ is the sub-$\spm$-module $\ker(H\phi\vert_E)$ of $E$ \begin{equation}\label{kerphi}
 \ker(H\phi\vert_E)(k_+)=\{x\in E(k_+)\mid H\phi(x)=*\}.
  \end{equation} 
\item[(ii)]The homotopy $\pik_0(\Gamma(\phi\vert_E))$ is the tolerant $\spm$-module $HB$ endowed with the relations 
	\begin{equation}\label{relk}
 \cR_k=\{(x,y)\in HB(k_+)\times HB(k_+)\mid x-y\in H\phi(E(k_+))\}\quad k\in\N.
  \end{equation}
  \end{enumerate}
\end{proposition}  
\begin{proof}
    By construction,  $\Gamma(\phi\vert_E)$ is the composite of  the functors $ H_\phi\vert_E:\crel \longrightarrow \Ses $ of \eqref{descript1sub} and $\gamma\circ\partial:\dop \longrightarrow\crel$.  One has $\partial[0]^*=(X,Y)$ where $X=Y=\{0,1\}$ with base point $0$. Thus  $\gamma\circ\partial[0]^*(k_+)=(X\wedge k_+,Y\wedge k_+)=(k_+,k_+)$ and by \eqref{descript1sub}
\[
H_\phi\vert_E(\gamma\circ\partial[0]^*)(k_+)=H_\phi\vert_E(X\wedge k_+,Y\wedge k_+)=HB(k_+).
\]
One has $\partial[1]^*=(X,Y)$, where $X=\{0,1,2\}$, $Y=\{0,2\}$ with base point $0$. Thus one has $\gamma\circ\partial[1]^*(k_+)=(X\wedge k_+,Y\wedge k_+)$ and $Y\wedge k_+=k_+$, while $(X\wedge k_+)/(Y\wedge k_+)=\{0,1\}\wedge k_+=k_+$ and by \eqref{descript1sub} it follows that
        \begin{equation}\label{pideg1}
 H_\phi\vert_E(\gamma\circ\partial[1]^*)(k_+)= E(k_+)\times HB(k_+).
\end{equation}
The boundaries $\partial_j: H_\phi\vert_E(\gamma\circ\partial[1]^*)\to H_\phi\vert_E(\gamma\circ\partial[0]^*)$ are obtained as in \cite{CCAtiyah} Proposition 4.11 
\begin{equation}\label{pi0}
\partial_0(\psi)=\psi_2, \ \ \partial_1(\psi)=H\phi(\psi_1)+\psi_2, \quad \forall \psi=(\psi_1,\psi_2)\in H_\phi\vert_E(\gamma\circ\partial[1]^*).
\end{equation}

$(i)$~The spherical condition $\partial_j(\psi)=*$ on  $\psi\in H_\phi\vert_E(\gamma\circ\partial[1]^*)(k_+)$ means that $\psi_2=0$ and $H\phi(\psi_1)=0$. Thus the solutions  correspond to $\ker(H\phi\vert_E)(k_+)$ as in \eqref{kerphi}. One shows as in \cite{CCAtiyah} Proposition 4.11 $(iii)$, that the relation of homotopy is the identity.\newline
$(ii)$~The relation $\cR_k$ on elements of $H_\phi\vert_E(\gamma\circ\partial[0]^*)(k_+)$ is defined as follows
\begin{equation}\label{pizeroR}
(\alpha,\beta)\in \cR_k\iff \exists 	\psi \in H_\phi\vert_E(\gamma\circ\partial[1]^*)(k_+)~\text{s.t.}~\partial_0(\psi)=\alpha \  \& \ \partial_1(\psi)=\beta.
\end{equation}
By \eqref{pi0} it coincides with \eqref{relk}.\vspace{.03in}

 \emph{Note} that this relation is a tolerance relation \ie is  symmetric since $E$ is a sub $\spm$-module of the $\spm$-module $HA$, so that 
\[
x-y\in H\phi(E(k_+))\iff y-x\in H\phi(E(k_+)).
\]
\end{proof}
\vspace{.1in}

The above general construction applies to the geometric adelic context and by implementing the action of $\Q^\times$ on adeles, we obtain the following variant of  Proposition~4.9 in \cite{CCAtiyah}.

 \begin{proposition}\label{bfhd} Let  $D$ be an Arakelov divisor on $\spzb$. Let $A=\Q \times \A_\Q$, $B=\A_\Q$ and $\alpha:A \to B$, $\alpha(q,a) = q+ a$. Let  $E= H\Q\times H\cO(D)$ be the sub $\spm$-module of $HA$ as in Proposition \ref{setup}. 
 	Then the  functor  
\begin{equation}\label{propgam0}
\bfh(D):=\Gamma(\alpha\vert_E): \dop \longrightarrow \smod
\end{equation}
defines a $\Gamma$-space
  that depends only on the  linear equivalence class of $D$.
 \end{proposition}
 
 \subsection{Proof of Proposition \ref{setup1}} \label{sectproofprop}
 
 With the notations of section \ref{secA} and Proposition \ref{bfhd}, one has $H^0(D):= \pik_1(\bfh(D))=\pik_1(\Gamma(\alpha\vert_E))$. Proposition \ref{proppi0}, $(i)$, gives 
 \[
 H^0(D)(k_+)=\{x\in E(k_+)\mid H\alpha(x)=*\}.
 \]
 An element $x\in E(k_+)$ is  a $k$-tuple $(q_j,a_j)$ with $q_j\in \Q$ for all $j$, and  $(a_j)\in H\cO(D)(k_+)$, where $H\cO(D):=H\cO(D)_f\times \Vert H\R\Vert_{e^a}$. The condition $H\alpha(x)=*$ means $q_j=-a_j$ for all $j$, so that $x$ is uniquely determined by the $k$-tuple $(q_j)$. Moreover, the allowed $k$-tuples are those for which   $\sum \vert q_j\vert \leq e^a$ and $q_j\in  \cO(D)_f$ for all $j$. One has  
 \[
 \Q\cap \cO(D)_f= H^0(\spz ,\cO(D)_f).
 \]
This proves $(i)$ of Proposition \ref{setup1}. The  statement $(ii)$ follows  from Proposition \ref{proppi0} $(ii)$.  Finally, $(iii)$ follows from Proposition \ref{bfhd}.

\vspace{.1in}




\end{appendix}

\begin{Backmatter}

\paragraph{Acknowledgments}
The second author is partially supported by the Simons Foundation collaboration grant n. 691493 and thanks IHES for the hospitality,  where part of this research was done.

\end{Backmatter}

\end{document}